\documentclass[11pt,reqno]{amsart}
\usepackage{indentfirst, amssymb,amsmath,amsthm, mathrsfs,cite,graphicx,float}
\newtheorem*{theoA}{Theorem A}
\newtheorem*{theoB}{Theorem B}
\newtheorem*{theoC}{Theorem C}

\newtheorem{theo}{Theorem}[section]
\newtheorem{lem}{Lemma}[section]

\newtheorem{rem}{Remark}[section]

\newtheorem{que}{Question}[section]

\newtheorem{open problem}{Open problem}[section]
\newcommand{\pa}{\partial}

\newcommand{\D}{\mathbb{D}}
\newcommand{\C}{\mathbb{C}}
\newcommand{\N}{\mathbb{N}}
\newcommand{\be}{\begin{equation}}
\newcommand{\ee}{\end{equation}}
\newcommand{\bs}{\begin{small}}
\newcommand{\es}{\end{small}}
\newcommand{\beas}{\begin{eqnarray*}}
\newcommand{\eeas}{\end{eqnarray*}}
\newcommand{\bea}{\begin{eqnarray}}
\newcommand{\eea}{\end{eqnarray}}
\renewcommand{\epsilon}{\varepsilon}
\numberwithin{equation}{section}

\begin{document}
\title[The Bohr's Phenomenon]{The Bohr's Phenomenon involving multiple Schwarz functions}
\author[V. Allu, R. Biswas and R. Mandal]{Vasudevarao Allu, Raju Biswas and Rajib Mandal}
\date{}
\address{Vasudevarao Allu, Department of Mathematics, School of Basic Science, Indian Institute of Technology Bhubaneswar, Bhubaneswar-752050, Odisha, India.}
\email{avrao@iitbbs.ac.in}
\address{Raju Biswas, Department of Mathematics, Raiganj University, Raiganj, West Bengal-733134, India.}
\email{rajubiswasjanu02@gmail.com}
\address{Rajib Mandal, Department of Mathematics, Raiganj University, Raiganj, West Bengal-733134, India.}
\email{rajibmathresearch@gmail.com}
\maketitle
\let\thefootnote\relax
\begin{abstract} 
The primary objective of this paper is to establish several sharp versions of Bohr inequalities for bounded analytic functions in the unit disk $\mathbb{D} := \{z\in\mathbb{C} : |z| < 1\}$ involving multiple Schwarz functions. 
Moreover, we obtain an improved version of the classical Rogosinski inequality for analytic functions in $\mathbb{D}$.
\end{abstract}
\section{Introduction and Preliminaries}
\noindent Let $H_\infty$ denote the class of all bounded analytic functions in the unit disk $\mathbb{D}:=\{z\in\C: |z|<1\}$ with the supremum norm $\Vert f\Vert_\infty:=\sup_{z\in\D}|f(z)|$. 
If $f\in H_\infty$ has the Taylor series expansion $f(z)=\sum_{n=0}^{\infty} a_nz^n$, then 
\bea\label{e2} \sum_{n=0}^{\infty}|a_n|r^n\leq \Vert f\Vert_\infty\quad\text{for}\quad|z|= r\leq\frac{1}{3}.\eea
Here, the quantity $1/3$ is known as Bohr radius and it cannot be improved. The inequality (\ref{e2}) is known as the classical Bohr inequality. 
In fact, H. Bohr \cite{B1914} showed the inequality (\ref{e2}) only for $r\leq 1/6$. However, subsequently, Weiner, Riesz, and Schur \cite{D1995} independently proved that (\ref{e2}) holds for $r\leq 1/3$. 
Note that, if $|f(z)|\leq 1$ in $\mathbb{D}$ such that $|f(z_0)|=1$ for some $z_0\in\mathbb{D}$, then $f(z)$ reduces to a unimodular constant function (see \cite[Strict Maximum 
Principle (Complex Version), P. 88]{G2000}).\\[2mm]
\indent In addition to the notion of the Bohr radius, there is another concept known as the Rogosinski radius \cite{R1923}, which is defined as follows:
Let $f(z)=\sum_{n=0}^{\infty}a_nz^n$ be analytic in $\mathbb{D}$ such that $|f(z)|<1$ in $\Bbb{D}$. Then, $\left|\sum_{n=0}^{N-1}a_nz^n\right|<1$ for every $N\geq 1$ in the disk $|z|<1/2$. The number $1/2$ is sharp. This inequality is called the classical Rogosinski inequality for bounded analytic 
functions. In 2017, Kayumov and Ponnusamy\cite{KP2017} considered the Bohr-Rogosinski sum $R_N^f(z)$ which is defined as follows:
\beas R_N^f(z):=|f(z)|+\sum_{n=N}^{\infty}|a_n||z|^n, \quad N\in\N.\eeas
It is evident that $|S_N(z)|=\left|f(z)-\sum_{n=N}^{\infty}a_nz^n\right|\leq R_N^f(z)$. Moreover, the Bohr-Rogosinski sum
$R_N^f(z)$ is related to the classical Bohr sum in which $N=1$ and $f(z)$ is replaced by $f(0)$. Let $f$ be an analytic function in $\D$ with $|f(z)|<1$
in $\D$. Kayumov and Ponnusamy \cite{KP2017} have defined the Bohr-Rogosinski radius as the largest number $ r_0\in(0, 1)$ such that the inequality $R_N^f(z)\leq 1$ holds for $|z|< r_0$. In \cite{KP2017}, Kayumov and Ponnusamy have obtained several results regarding to the Bohr-Rogosinski radius.\\[2mm]
Prior to proceeding with the discussion, it is necessary to establish certain notations.
Let $\mathcal{B}=\{f\in H_\infty :\Vert f\Vert_\infty\leq 1\}$ and $m\in\N$, let 
\beas\mathcal{B}_m=\left\{\omega\in\mathcal{B}: \omega(0)=\omega'(0)=\cdots=\omega^{(m-1)}(0)=0\quad\text{and}\quad \omega^{(m)}(0)\not=0\right\}.\eeas
\noindent Let $f$ be an analytic function in $\mathbb{D}$ and $\mathbb{D}_ r:=\{z\in\mathbb{C}: |z|< r \}$ for $0< r<1$. Let $S_ r$ denote the planar integral  
\beas S_ r=\int_{\mathbb{D}_ r} |f'(z)|^2 dA(z).\eeas
Furthermore, if $f(z)=\sum_{n=0}^\infty a_n z^n$, then it is well-known that $S_ r/\pi=\sum_{n=1}^\infty n|a_n|^2  r^{2n}$ and if $f$ is univalent, then $S_ r$ is the area of the image $f(\mathbb{D}_ r)$. \\[1mm]
\indent In 2018, Kayumov and Ponnusamy \cite{KP2018} obtained several improved versions of Bohr's inequality for the bounded analytic functions in $\mathbb{D}$.
Based on the initiation of Kayumov and Ponnusamy \cite{KP2017, KP2018}, a number of authors have studied several versions of Bohr-type inequalities (see \cite{AKP2020,LSX2018,KKP2021}).\\[1mm] 
 Recently, Liu {\it et al.} \cite{LLP2021} have obtained the following refined version of Bohr inequalities.
\begin{theoA}\cite{LLP2021}
Let $f(z)=\sum_{n=0}^\infty a_nz^n$ be analytic in $\mathbb{D}$ and $|f(z)|<1$ in $\mathbb{D}$. Then,
\beas |f(z)|+|f'(z)| r+\sum_{n=2}^\infty |a_n| r^n+\left(\frac{1}{1+|a_0|}+\frac{ r}{1- r}\right)\sum_{n=1}^\infty |a_n|^2  r^{2n}\leq 1\eeas
for $|z|=r\leq (\sqrt{17}-3)/4$. The number $(\sqrt{17}-3)/4$ is best possible. \end{theoA}
\begin{theoB}\cite{LLP2021}  Let $f(z)=\sum_{n=0}^{\infty} a_nz^n$ be analytic in $\mathbb{D}$ and $|f(z)|\leq 1$. Then,
\beas |f(z)|^2+|f'(z)| r+\sum_{n=2}^\infty |a_n| r^n+\left(\frac{1}{1+|a_0|}+\frac{ r}{1- r}\right)\sum_{n=1}^\infty |a_n|^2  r^{2n}\leq 1\eeas
for $|z|=r\leq r_0$, where $r_0\approx 0.385795$ is the unique positive root of the equation $1-2r-r^2-r^3-r^4=0$. The number $r_0$ is best possible.\end{theoB}
\begin{theoC}\cite{LLP2021}  Let $f(z)=\sum_{n=0}^{\infty} a_nz^n$ be analytic in $\mathbb{D}$ and $|f(z)|\leq 1$. Then,
\beas \sum_{n=0}^\infty |a_n| r^n+\left(\frac{1}{1+|a_0|}+\frac{r}{1- r}\right)\sum_{n=1}^\infty |a_n|^2  r^{2n}+\frac{8}{9}\left(\frac{S_r}{\pi}\right)\leq 1\eeas
for $|z|=r\leq 1/3$. The numbers $1/3$ and $8/9$ cannot be improved.\end{theoC}
In addition to these, a number of authors have studied several other aspects of Bohr's inequality.
We refer to \cite{ABM2024,AKP2019,AAH2022,AH2021,1AH2021,AH2022,BDK2004,BB2004,EPR2019,EPR2021,FR2010,HLP2020,IKP2020,1KP2018,LP2019,MBG2024,PW2020} and the references listed therein for an in-depth investigation on Bohr radius.\\[2mm]
\noindent In light of the aforementioned findings, several questions naturally arise with regard to this study.
\begin{que}\label{Q1}
Is it possible to establish several improved and refined versions of the Bohr inequality for bounded analytic functions that involve more than one Schwarz function?
\end{que}
\begin{que}\label{Q2}Can an improved version of the classical Rogosinski inequality be formulated that includes multiple Schwarz functions and substitutes the Taylor series coefficients with the derivatives of the analytic function? \end{que}
The purpose of this paper is primarily to provide the affirmative answers to Questions \ref{Q1} and \ref{Q2}.
\section{Some lemmas}
The following lemmas are essential to prove the main results of this paper.
\begin{lem}\label{lem1} \cite[Pick's invariant form of Schwarz's lemma]{K2006} Suppose $f$ is analytic in $\mathbb{D}$ with $|f(z)|\leq1$, then 
\beas |f(z)|\leq \frac{|f(0)|+|z|}{1+|f(0)||z|}\quad\text{for}\quad z\in\mathbb{D}.\eeas\end{lem}
\begin{lem}\cite{DF2008}\label{lem2} Suppose $f$ is analytic in $\mathbb{D}$ with $|f(z)|\leq1$, then we have 
\beas \frac{\left|f^{(n)}(z)\right|}{n!}\leq \frac{1-|f(z)|^2}{(1-|z|)^{n-1}(1-|z|^2)}\quad{and}\quad |a_n|\leq 1-|a_0|^2\;\;\text{for}\quad n\geq 1\quad \text{and}\quad |z|<1.\eeas\end{lem}
\begin{lem}\cite{LLP2021}\label{lem5} Suppose $f$ is analytic in $\mathbb{D}$ with $|f(z)|\leq1$, then for any $N\in\mathbb{N}$, the following inequality holds:
\beas\sum_{n=N}^\infty |a_n| r^n+\text{sgn}(t)\sum_{n=1}^t |a_n|^2 \frac{ r^N}{1- r}+\left(\frac{1}{1+|a_0|}+\frac{ r}{1- r}\right)\sum_{n=t+1}^\infty |a_n|^2  r^{2n}\leq \frac{(1-|a_0|^2) r^N}{1- r}\eeas
for $ r\in[0,1)$, where $t=\lfloor (N-1)/2\rfloor$ and $\lfloor x\rfloor$ denotes the largest integer not exceeding the real number $x$.
\end{lem}
\section{Main Results}
In the following result, we establish the sharp refined version of the Bohr inequality involving several Schwarz functions.
\begin{theo}\label{Th1}
Suppose that $f(z)=\sum_{n=0}^\infty a_n z^n$ is analytic in $\D$ with $|f(z)|\leq 1$ in $\D$ and $\omega_k\in\mathcal{B}_k$ for $k\geq 1$. Then
\beas&& |f(\omega_m(z))|+|\omega_p(z)||f'(\omega_m(z))|+\sum_{n=2}^\infty |a_n||\omega_k(z)|^n\\
&&+\left(\frac{1}{1+|a_0|}+\frac{|\omega_k(z)|}{1-|\omega_k(z)|}\right)\sum_{n=1}^\infty |a_n|^2 |\omega_k(z)|^{2n}\leq 1\eeas 
for $r\leq R_{m,p,k}\leq r_{m,p}$, where $r_{m,p}\in(0,1)$ is the unique root of the equation $r^{2m}+2r^p-1=0$ and $R_{m,p,k}\in(0,1)$ is the smallest positive root of the equation 
\beas \frac{2r^p}{1+r^m}+\frac{2 r^{2k}(1+r^m)}{1- r^k} -(1-r^m)=0.\eeas
The number $R_{m,p,k}$ cannot be improved.
\end{theo} 
\begin{proof}
Since $f(z)=\sum_{n=0}^\infty a_n z^n$ is analytic in $\D$ with $|f(z)|\leq 1$ in $\D$, in view of \textrm{Lemmas \ref{lem1}} and \ref{lem2}, we have $|a_n|\leq 1-|a_0|^2$ for $n\geq 1$ and 
\beas |f(z)|\leq \frac{|a_0|+|z|}{1+|a_0||z|}\quad\text{and}\quad |f'(z)|\leq \frac{1-|f(z)|^2}{1-|z|^2}.\eeas
Let $|a_0|=a\in[0,1]$. Since $\omega_m\in\mathcal{B}_m$, in view of the classical Schwarz Lemma, we have $|\omega_m(z)|\leq |z|^m$. 
Let $F_1(x)=(a+x)/(1+ax)$ and $F_2(x)=x/(1-x)$, where $0\leq x\leq x_0(\leq1)$ and $a\geq 0$. It is evident that both $F_1(x)$ and $F_2(x)$ are monotonically increasing functions of $x$, and it follows that $F_i(x)\leq F_i(x_0)$ $(i=1,2)$.
Thus
\bea\label{f3} |f(\omega_m(z))|\leq \frac{a+|\omega_m(z)|}{1+a|\omega_m(z)|}\leq \frac{a+r^m}{1+ar^m},\quad|z|=r<1.\eea
In view of \textrm{Lemma \ref{lem5}}, we have
\bea\label{f1} \sum_{n=2}^\infty |a_n| r^n+\left(\frac{1}{1+a}+\frac{r}{1- r}\right)\sum_{n=1}^\infty |a_n|^2  r^{2n}\leq \frac{(1-a^2) r^2}{1- r}.\eea
Using (\ref{f1}), we have
\bea\label{f4} &&\sum_{n=2}^\infty |a_n||\omega_k(z)|^n+\left(\frac{1}{1+|a_0|}+\frac{|\omega_k(z)|}{1-|\omega_k(z)|}\right)\sum_{n=1}^\infty |a_n|^2 |\omega_k(z)|^{2n}\nonumber\\
&\leq&\sum_{n=2}^\infty |a_n|r^{nk}+\left(\frac{1}{1+a}+\frac{r^k}{1-r^k}\right)\sum_{n=1}^\infty |a_n|^2 r^{2kn}\leq\frac{(1-a^2) r^{2k}}{1- r^k}.\eea
Let $G(t)=t+\alpha(1-t^2)$, where $0\leq t\leq t_0(\leq1)$ and $\alpha\geq 0$. Then, $G(t)\leq G(t_0)$ for $0\leq \alpha\leq 1/2$ (see \cite{RB2024}). Therefore
\beas && |f(\omega_m(z))|+|\omega_p(z)||f'(\omega_m(z))|+\sum_{n=2}^\infty |a_n||\omega_k(z)|^n\\
&&+\left(\frac{1}{1+|a_0|}+\frac{|\omega_k(z)|}{1-|\omega_k(z)|}\right)\sum_{n=1}^\infty |a_n|^2 |\omega_k(z)|^{2n}\\[2mm]
&\leq& |f(\omega_m(z))|+\frac{r^p}{1-r^{2m}}\left(1-|f(\omega_m(z))|^2\right)+\frac{(1-a^2) r^{2k}}{1- r^k}\\[2mm]
&\leq& \frac{a+r^m}{1+ar^m}+\frac{r^p}{1-r^{2m}}\left(1-\left(\frac{a+r^m}{1+ar^m}\right)^2\right)+\frac{(1-a^2) r^{2k}}{1- r^k}\\[2mm]
&=&1+\frac{(1-a)G_1(a,r)}{1+ar^m}\eeas
for $r\in[0,r_{m,p}]$, where $r_{m,p}\in(0,1)$ is the unique root of the equation $r^{2m}+2r^p-1=0$ and 
\beas G_1(a,r)=\frac{(1+a)r^p}{(1+ar^m)}+\frac{(1+a)(1+ar^m) r^{2k}}{1- r^k} -(1-r^m).\eeas
Differentiating $G_1(a,r)$ partially with respect to $a$, we obtain
\beas \frac{\pa }{\pa a}G_1(a,r)=\frac{r^p(1-r^m)}{(1+a r^m)^2}+\frac{r^{2 k}(1+r^m+2 a r^m)}{1 - r^k}\geq 0.\eeas
Therefore, $G_1(a,r)$ is a monotonically increasing function of $a\in[0,1]$ and hence, we have 
\beas G_1(a,r)\leq G_1(1,r)=\frac{2r^p}{1+r^m}+\frac{2 r^{2k}(1+r^m)}{1- r^k} -(1-r^m)\leq 0,\eeas
for $r\leq R_{m,p,k}$, where $R_{m,p,k}\in(0,1)$ is the smallest positive root of the equation 
\bea\label{f2} G_2(r):=\frac{2r^p}{1+r^m}+\frac{2 r^{2k}(1+r^m)}{1- r^k} -(1-r^m)=0.\eea
We claim that $R_{m,p,k}\leq r_{m,p}$. For $r>r_{m,p}$, we have $r^{2m}+2r^p-1>0$ and 
\beas \frac{2r^p}{1+r^m}+\frac{2 r^{2k}(1+r^m)}{1- r^k} -(1-r^m)=\frac{r^{2m}+2r^p-1}{1+r^m}+\frac{2 r^{2k}(1+r^m)}{1- r^k}>0.\eeas
This shows that $R_{m,p,k}\leq r_{m,p}$.\\[2mm]
\indent To prove the sharpness of the result, we consider the following functions 
\beas f_1(z)=\frac{a+z}{1+az}=A_0+\sum_{n=1}^\infty A_n z^n\quad\text{and}\quad \omega_m(z)=z^m\quad\text{for}\quad m\geq 1,\eeas
where $A_0=a\in[0,1)$, $A_n=(1-a^2)(-a)^{n-1}$ for $n\geq 1$.
Thus, for $z=r$, we have 
\beas&& |f_1(r^m)|+|r^p||f_1'(r^m)|+\sum_{n=2}^\infty |A_n||r|^{nk}+\left(\frac{1}{1+|A_0|}+\frac{|r^k|}{1-|r^k|}\right)\sum_{n=1}^\infty |A_n|^2 |r|^{2nk}\\[2mm]
&&=\frac{a+ r^m}{1+a r^m}+\frac{(1-a^2)r^{p}}{(1+a r^m)^2}+\frac{(1-a^2)a r^{2 k}}{1-a r^k}+\frac{(1-a^2)^2 r^{2k}(1+a r^k)}{(1+a)(1- r^k)}\sum_{n=1}^\infty (a r^k)^{2(n-1)}\\[2mm]
&&=\frac{a+ r^m}{1+a r^m}+\frac{(1-a^2)r^{p}}{(1+a r^m)^2}+\frac{(1-a^2)a r^{2 k}}{1-a r^k}+\frac{1+a r^k}{(1+a)(1- r^k)}\frac{(1-a^2)^2 r^{2k}}{1-a^2 r^{2k}}\\[2mm]
&&=1+(1-a)G_3(a, r),\eeas
where 
\beas G_3(a, r)=\frac{(1+a) r^{p}}{(1+a r^m)^2}+\frac{(1+a)a r^{2 k}}{1-a r^k}+\frac{(1-a^2) r^{2k}}{(1-a r^{k})(1- r^k)}-\frac{1-r^m}{1+ar^m}.\eeas
It is evident that
\beas \lim_{a\to1^-} G_3(a, r)
=\frac{2 r^{p}}{(1+ r^m)^2}+\frac{2 r^{2 k}}{1- r^k}-\frac{1-r^m}{1+r^m}>0 \eeas
for $r>R_{m,p,k}$, where $R_{m,p,k}$ is the smallest positive root of the equation (\ref{f2}) in $[0,r_{m,p}]$. This shows that the radius $R_{m,p,k}$ is the best possible. This completes the proof.
\end{proof}
\begin{rem}
We discuss some special cases of \textrm{Theorem \ref{Th1}} as well as a number of helpful observations.
\begin{enumerate}
\item[(i)] Setting $p=m$, $k=1$ and $\omega_1(z)=z$ in \textrm{Theorem \ref{Th1}} gives the first part of \textrm{Theorem 3} of \cite{HLP2020}.
\item[(ii)] Setting $p=k=1$ and $\omega_1(z)=z$ in \textrm{Theorem \ref{Th1}} gives the first part of \textrm{Theorem 4} of \cite{HLP2020}.
\item[(iii)] Setting $m=p=k=1$ and $\omega_1(z)=z$ in \textrm{Theorem \ref{Th1}} gives \textrm{Theorem A}.
\end{enumerate}
\end{rem}
\noindent In Table \ref{tab1} and Figures \ref{fig1} and \ref{fig2}, we obtain the values of $R_{m,p,k}$ and $r_{m,p}$ for certain values of $m,p,k\in\mathbb{N}$.
\begin{table}[H]
\centering
\begin{tabular}{*{5}{|c}|}
\hline
$m$ &p&k&$R_{m,p,k}$&$r_{m,p}$\\
\hline
1&1&1&$(\sqrt{17}-3)/4$&0.414214\\
\hline
3&3&1&0.441112&0.745432\\
\hline
2&3&2&0.567006&0.716673\\
\hline
5&30&10&0.88777&0.948565\\
\hline
30&20&10&0.914967&0.961223\\
\hline
\end{tabular}
\caption{$R_{m,p,k}$ is the smallest positive root of the equation (\ref{f2}) in $[0, r_{m,p}]$}
\label{tab1}\end{table}
\vspace{1cm}
\begin{figure}[h]
\begin{minipage}[c]{1.0\linewidth}
\centering
\includegraphics[scale=0.8]{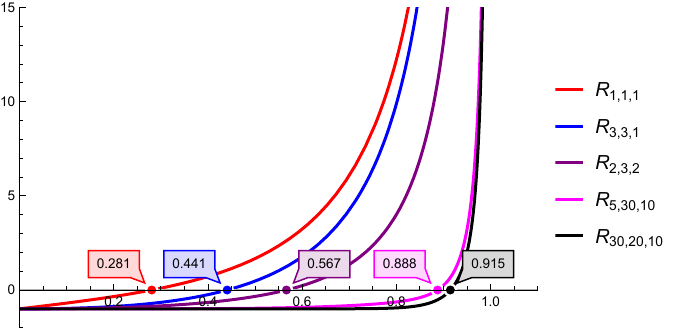}
\caption{Location of $R_{m,p,k}$ in $(0,1)$}
\label{fig1}
\end{minipage}\hspace{0.5cm}
\begin{minipage}[c]{1.0\linewidth}
\centering
\includegraphics[scale=0.8]{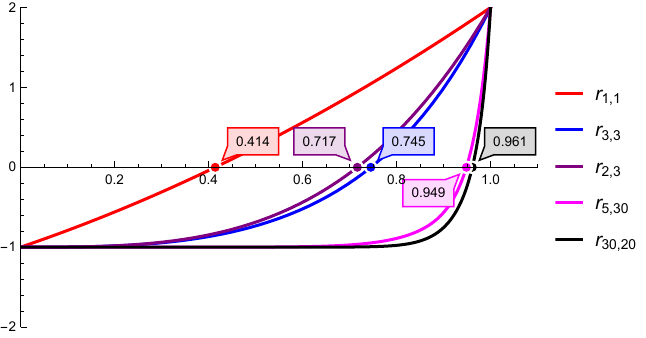}
\caption{Location of $r_{m,p}$ in $(0,1)$}
\label{fig2}
\end{minipage}
\end{figure}
In the following result, we establish the sharp refined version of the Bohr inequality in the context of several Schwarz functions.
\begin{theo}\label{Th2}
Suppose that $f(z)=\sum_{n=0}^\infty a_n z^n$ is analytic in $\D$ with $|f(z)|\leq 1$ in $\D$ and $\omega_k\in\mathcal{B}_k$ for $k\geq 1$. Then
\beas &&|f(\omega_m(z))|^2+|\omega_p(z)||f'(\omega_m(z))|+\sum_{n=2}^\infty |a_n||\omega_k(z)|^n\\&&+\left(\frac{1}{1+|a_0|}
+\frac{|\omega_k(z)|}{1-|\omega_k(z)|}\right)\sum_{n=1}^\infty |a_n|^2 |\omega_k(z)|^{2n}\leq 1\eeas 
for $r\leq R_{2,m,p,k}$, where $R_{2,m,p,k}\in(0,r_{2,m,p})$ is the smallest positive root of
the equation $r^{2k}/(1- r^k) -(1-r^{2m}-r^p)/(1+r^m)^2=0$ and $r_{2,m,p}\in(0,1)$ is the unique root of the equation $1-r^{2m}-r^p=0$. The number $R_{2,m,p,k}$ cannot be improved.
\end{theo} 
\begin{proof}Using similar argument as in the proof of \textrm{Theorem \ref{Th1}}, we obtain (\ref{f3}) and (\ref{f4}). Therefore, we have
\beas && |f(\omega_m(z))|^2+|\omega_p(z)||f'(\omega_m(z))|+\sum_{n=2}^\infty |a_n||\omega_k(z)|^n\\
&&+\left(\frac{1}{1+|a_0|}+\frac{|\omega_k(z)|}{1-|\omega_k(z)|}\right)\sum_{n=1}^\infty |a_n|^2 |\omega_k(z)|^{2n}\\[2mm]
&\leq& |f(\omega_m(z))|^2+\frac{r^p}{1-r^{2m}}\left(1-|f(\omega_m(z))|^2\right)+\frac{(1-a^2) r^{2k}}{1- r^k}\eeas
\beas&\leq&\left(1-\frac{r^p}{1-r^{2m}}\right) \left(\frac{a+r^m}{1+ar^m}\right)^2+\frac{r^p}{1-r^{2m}}+\frac{(1-a^2) r^{2k}}{1- r^k}\\[2mm]
&=&1+\frac{1-r^{2m}-r^p}{1-r^{2m}}\left( \left(\frac{a+r^m}{1+ar^m}\right)^2-1\right)+\frac{(1-a^2) r^{2k}}{1- r^k}\\[2mm]
&=&1+(1-a^2)G_4(a,r),\eeas
where the second inequality holds for such $r\in[0,1]$ satisfying $1-r^p/(1-r^{2m})\geq 0$, {\it i.e.,} for $r\in[0,r_{2,m,p}]$, where $r_{2,m,p}\in(0,1)$ is the unique root of the equation $1-r^{2m}-r^p=0$ and 
\beas G_4(a,r)=-\frac{(1-r^{2m}-r^p)}{(1+ar^m)^2}+\frac{r^{2k}}{1- r^k}.\eeas
It is evident that the function $G_4(a,r)$ is a monotonically increasing function of $a$ and it follows that 
\beas G_4(a,r)\leq G_4(1,r)=-\frac{(1-r^{2m}-r^p)}{(1+r^m)^2}+\frac{r^{2k}}{1- r^k}\leq 0\eeas
for $r\leq R_{2,m,p,k}$, where $R_{2,m,p,k}$ is the smallest positive root of the equation 
\bea\label{f5} -\frac{(1-r^{2m}-r^p)}{(1+r^m)^2}+\frac{r^{2k}}{1- r^k}=0.\eea
It is clear that $R_{2,m,p,k}\leq r_{2,m,p}$.\\[2mm] 
\indent To prove the sharpness of the result, we consider the following functions 
\beas f_2(z)=\frac{a+z}{1+az}=A_0+\sum_{n=1}^\infty A_n z^n\quad\text{and}\quad \omega_m(z)=z^m\quad\text{for}\quad m\geq 1,\eeas
where $A_0=a\in[0,1)$, $A_n=(1-a^2)(-a)^{n-1}$ for $n\geq 1$.
Thus, for $z=r$, we have 
\beas&& |f_2(r^m)|^2+|r^p||f_2'(r^m)|+\sum_{n=2}^\infty |A_n||r|^{nk}+\left(\frac{1}{1+|A_0|}+\frac{|r^k|}{1-|r^k|}\right)\sum_{n=1}^\infty |A_n|^2 |r|^{2nk}\\[2mm]
&=&\left(\frac{a+ r^m}{1+a r^m}\right)^2+\frac{(1-a^2)r^{p}}{(1+a r^m)^2}+\frac{(1-a^2)a r^{2 k}}{1-a r^k}+\frac{(1-a^2)^2 r^{2k}(1+a r^k)}{(1+a)(1- r^k)}\sum_{n=1}^\infty (a r^k)^{2(n-1)}\\[2mm]
&=&1-\frac{(1-a^2)(1-r^{2m})}{(1+a r^m)^2}+\frac{(1-a^2)r^{p}}{(1+a r^m)^2}+\frac{(1-a^2)a r^{2 k}}{1-a r^k}+\frac{(1+a r^k)}{(1+a)(1- r^k)}\frac{(1-a^2)^2 r^{2k}}{1-a^2 r^{2k}}\\[2mm]
&=&1+(1-a^2)G_5(a, r),\eeas
where 
\beas G_5(a, r)=-\frac{(1-r^{2m})}{(1+a r^m)^2}+\frac{r^{p}}{(1+a r^m)^2}+\frac{a r^{2 k}}{1-a r^k}+\frac{(1+a r^k)}{(1+a)(1- r^k)}\frac{(1-a^2) r^{2k}}{1-a^2 r^{2k}}.\eeas
It is evident that
\beas \lim_{a\to1^-} G_5(a, r)
=\frac{r^{2m}+r^{p}-1}{(1+r^m)^2}+\frac{r^{2 k}}{1-r^k}>0 \eeas
for $r>R_{2,m,p,k}$, where $R_{2,m,p,k}$ is the smallest positive root of the equation (\ref{f5}) in $[0,r_{2,m,p}]$. This shows that the radius $R_{2,m,p,k}$ is the best possible. This completes the proof.
\end{proof}
\begin{rem}
We discuss some special cases of \textrm{Theorem \ref{Th2}} as well as a number of helpful observations.
\begin{enumerate}
\item[(i)] Setting $p=m$, $k=1$ and $\omega_1(z)=z$ in \textrm{Theorem \ref{Th2}} gives the second part \textrm{Theorem 3} of \cite{HLP2020}.
\item[(ii)] Setting $p=k=1$ and $\omega_1(z)=z$ in \textrm{Theorem \ref{Th2}} gives the second part of \textrm{Theorem 4} of \cite{HLP2020}.
\item[(iii)] Setting $m=p=k=1$ and $\omega_1(z)=z$ in \textrm{Theorem \ref{Th2}} gives \textrm{Theorem C}.
\end{enumerate}
\end{rem}
\noindent In Table \ref{tab2} and Figures \ref{fig3} and \ref{fig4}, we obtain the values of $R_{2,m,p,k}$ and $r_{2,m,p}$ for certain values of $m,p,k\in\mathbb{N}$.
\begin{table}[H]
\centering
\begin{tabular}{*{5}{|c}|}
\hline
$m$ &p&k&$R_{2,m,p,k}$&$r_{2,m,p}$\\
\hline
1&1&1&$0.385795$&0.618034\\
\hline
3&3&1&0.535687&0.8518\\
\hline
2&3&2&0.640675&0.819173\\
\hline
5&30&10&0.906065&0.962497\\
\hline
30&20&10&0.936066&0.981069\\
\hline
\end{tabular}
\caption{$R_{2,m,p,k}$ is the smallest positive root of the equation (\ref{f5}) in $[0, r_{2,m,p}]$}
\label{tab2}\end{table}
\begin{figure}[H]
\begin{minipage}[c]{1\linewidth}
\centering
\includegraphics[scale=0.8]{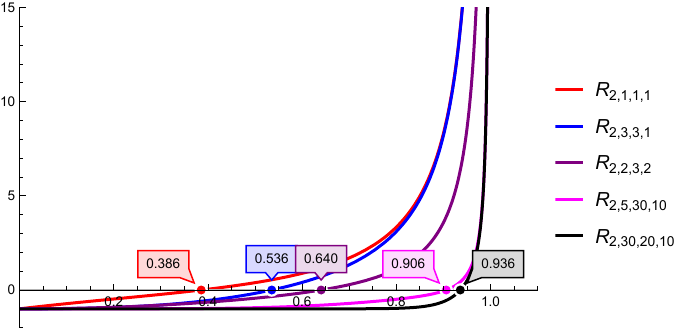}
\caption{Location of $R_{2,m,p,k}$ in $(0,1)$}
\label{fig3}
\end{minipage}\hspace{0.5cm}
\begin{minipage}[c]{1\linewidth}
\centering
\includegraphics[scale=0.8]{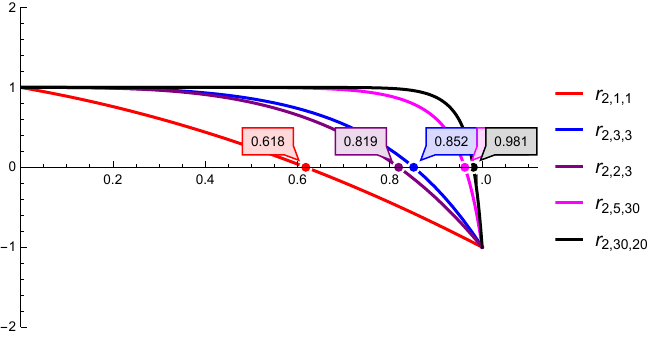}
\caption{Location of $r_{2,m,p}$ in $(0,1)$}
\label{fig4}
\end{minipage}
\end{figure}
\begin{theo}\label{Th6}
Suppose that $f(z)=\sum_{n=0}^\infty a_n z^n$ is analytic in $\D$ with $|f(z)|\leq 1$ in $\D$ and $\omega_k\in\mathcal{B}_k$ for $k\geq 1$. Then
\beas &&\sum_{n=0}^\infty |a_n||\omega_k(z)|^n+\left(\frac{1}{1+|a_0|}+\frac{|\omega_k(z)|}{1-|\omega_k(z)|}\right)\sum_{n=1}^\infty |a_n|^2 |\omega_k(z)|^{2n}\\
&&+\frac{(1-r_k^{2m})^2}{8r_k^{2m}}\sum_{n=1}^\infty n|a_n|^2 |\omega_m(z)|^{2n}\leq 1\eeas 
for $r\leq r_{k}$, where $r_{k}\in(0,1)$ is the unique root of the equation $3r^k-1=0$.
The numbers $r_{k}$ and $(1-r_k^{2m})^2/(8r_k^{2m})$ cannot be improved.
\end{theo}
\begin{proof} Using similar argument as in the proof of \textrm{Theorems \ref{Th1}} and in view of \textrm{Lemmas \ref{lem2}} and \ref{lem5}, we have $|a_n|\leq 1-a^2$ for $n\geq 1$ and
\beas  \sum_{n=1}^\infty |a_n||\omega_k(z)|^n+\left(\frac{1}{1+|a_0|}+\frac{|\omega_k(z)|}{1-|\omega_k(z)|}\right)\sum_{n=1}^\infty |a_n|^2 |\omega_k(z)|^{2n}\leq \frac{(1-a^2) r^k}{1- r^k},\eeas
where $|a_0|=a\in[0, 1]$.
Therefore,
\beas&& \sum_{n=0}^\infty |a_n||\omega_k(z)|^n+\left(\frac{1}{1+|a_0|}
+\frac{|\omega_k(z)|}{1-|\omega_k(z)|}\right)\sum_{n=1}^\infty |a_n|^2 |\omega_k(z)|^{2n}+\lambda\sum_{n=1}^\infty n|a_n|^2 |\omega_m(z)|^{2n}\\[2mm]
&\leq& a+\frac{(1-a^2) r^k}{1- r^k}+\lambda(1-a^2)^2 \sum_{n=1}^\infty n r^{2mn}\\
&\leq& a+\frac{(1-a^2) r^k}{1- r^k}+\lambda(1-a^2)^2\frac{r^{2m}}{(1-r^{2m})^2}=1+G_6(a,r),\eeas
where 
\beas G_6(a,r)=\frac{(1-a^2)}{2}\left(1+\left(\frac{2r^k}{1-r^k}-1\right)+2\lambda(1-a^2)\frac{r^{2m}}{(1-r^{2m})^2}-\frac{2}{1+a}\right).\eeas
Differentiating partially $G_6(a,r)$ with respect to $r$, we obtain
\beas \frac{\pa}{\pa r}G_6(a,r)=\frac{(1-a^2) k r^{k-1}}{(1- r^k)^2}+\lambda(1-a^2)^2\frac{2 m r^{2m-1}(1+r^{2m})}{(1-r^{2m})^3}\geq 0.\eeas
Therefore, $G_6(a,r)$ is a monotonically increasing function of $r$ in $[0, 1]$ and it follows that $G_6(a,r)\leq G_6(a,r_k)$ for $r\leq r_k$, where $r_k\in(0, 1)$ is the unique positive root of the equation $3r^k-1=0$.
Therefore, $G_6(a,r_k)=(1-a^2)G_7(a)/2$, where
\beas G_7(a)=1+2\lambda(1-a^2)\frac{r_k^{2m}}{(1-r_k^{2m})^2}-\frac{2}{1+a}.\eeas
Note that 
\beas G_7(0)=2\lambda\frac{r_k^{2m}}{(1-r_k^{2m})^2}-1\quad\text{and}\quad G_7(1)=0.\eeas
Differentiating $G_7(a)$ with respect to $a$, we obtain
\beas G_7'(a)=\frac{2}{(1+a)^2}\left(1-2\lambda \frac{r_k^{2m}}{(1-r_k^{2m})^2}a(1+a)^2\right)
\geq \frac{2}{(1+a)^2}\left(1-\frac{8\lambda r_k^{2m}}{(1-r_k^{2m})^2}\right)\geq 0\eeas
for $\lambda\leq (1-r_k^{2m})^2/(8r_k^{2m})$. Therefore, $G_7(a)$ is a monotonically increasing function of $a$, whenever $\lambda\leq (1-r_k^{2m})^2/(8r_k^{2m})$ and it follows that $G_7(a)\leq 0$ for $a\in[0,1]$. Therefore,
\beas&& \sum_{n=0}^\infty |a_n||\omega_k(z)|^n+\left(\frac{1}{1+|a_0|}+\frac{|\omega_k(z)|}{1-|\omega_k(z)|}\right)\sum_{n=1}^\infty |a_n|^2 |\omega_k(z)|^{2n}\\
&&+\frac{(1-r_k^{2m})^2}{8r_k^{2m}}\sum_{n=1}^\infty n|a_n|^2 |\omega_m(z)|^{2n}\leq 1\eeas
for $r\leq r_k$, where $r_k\in(0, 1)$ is the unique positive root of the equation $3r^k-1=0$.\\[2mm]
\indent To prove the sharpness of the result, we consider the following functions 
\beas f_3(z)=\frac{a+z}{1+az}=A_0+\sum_{n=1}^\infty A_n z^n\quad\text{and}\quad \omega_m(z)=z^m\quad\text{for}\quad m\geq 1,\eeas
where $A_0=a\in[0,1)$, $A_n=(1-a^2)(-a)^{n-1}$ for $n\geq 1$.
Thus, we have 
\beas&&\sum_{n=0}^\infty |A_n||r|^{nk}+\left(\frac{1}{1+|A_0|}+\frac{|r^k|}{1-|r^k|}\right)\sum_{n=1}^\infty |A_n|^2 |r|^{2nk}\\
&&+\frac{(1-r_k^{2m})^2}{8r_k^{2m}}\sum_{n=1}^\infty n|A_n|^2|r^m|^{2n}\\[2mm]
&=&a+\frac{(1-a^2)r^{ k}}{1-a r^k}+\frac{(1-a^2)^2 r^{2k}(1+a r^k)}{(1+a)(1- r^k)}\sum_{n=1}^\infty (a r^k)^{2(n-1)}\\
&&+\frac{(1-r_k^{2m})^2}{8r_k^{2m}}(1-a^2)^2r^{2m}\sum_{n=1}^\infty n (ar^{m})^{2(n-1)} \\[2mm]
&=&a+\frac{(1-a^2) r^{ k}}{1-a r^k}+\frac{1+a r^k}{(1+a)(1- r^k)}\frac{(1-a^2)^2 r^{2k}}{1-a^2 r^{2k}}+\frac{(1-r_k^{2m})^2}{8r_k^{2m}}\frac{(1-a^2)^2r^{2m}}{(1-a^2r^{2m})^2}\\[2mm]
&=&1+(1-a)G_8(a, r),\eeas
where 
\beas G_8(a, r)=\frac{(1+a) r^{ k}}{1-a r^k}+\frac{(1-a^2) r^{2k}}{(1- r^k)(1-a r^{k})}+\frac{(1-a^2)(1+a)(1-r_k^{2m})^2r^{2m}}{8r_k^{2m}(1-a^2r^{2m})^2}-1.\eeas
Differentiating partially $G_8(a, r)$ with respect to $r$, we obtain
\beas&& \frac{\pa}{\pa r} G_8(a, r)=\frac{(1+a) k r^{k-1}}{(1-a r^k)^2}+\frac{(1-a^2)(1+a)(1-r_k^{2m})^2}{8r_k^{2m}}\frac{2 m r^{2 m-1} \left(a^2 r^{2 m}+1\right)}{\left(1-a^2 r^{2 m}\right)^3}\\
&&+(1-a^2)\left(\frac{2kr^{2k-1}}{(1- r^k)(1-a r^{k})}+\frac{k r^{3k-1}}{\left(1-r^k\right)^2(1-a r^{k})}+\frac{a k r^{3k-1}}{(1- r^k)\left(1-a r^k\right)^2}\right)>0\eeas
for $r\in(0,1)$. Therefore, $G_8(a, r)$ is a strictly increasing function of $r\in(0,1)$. For $r>r_k$, we have $G_8(a, r)>G_8(a, r_k)$. It is evident that
\beas \lim_{a\to1^-}G_8(a,r_k)=\frac{2r_k^{ k}}{1-r_k^k}-1=0.\eeas
This shows that the numbers $r_k$ and $(1-r_k^{2m})^2/(8r_k^{2m})$ are the best possible. This completes the proof.
\end{proof}
\begin{rem}
Setting $k=m=1$ and $\omega_1(z)=z$ in \textrm{Theorem \ref{Th6}} gives \textrm{Theorem C}. In this case $r_k=1/3$ and $(1-r_k^{2m})^2/(8r_k^{2m})=8/9$. An observation shows that $(1-r_k^{2m})^2/(8r_k^{2m})=8/9$, whenever $k=m$.
\end{rem}
\noindent In \textrm{Table \ref{tab3}}, we obtain the values of $r_k$ and $R_{m}(r_k):=(1-r_k^{2m})^2/(8r_k^{2m})$ for certain values of $m,k\in\mathbb{N}$.
\begin{table}[H]
\centering
\begin{tabular}{*{4}{|c}|}
\hline
$m$&k&$r_k$&$R_{m}(r_k)$\\
\hline
2&1&1/3&800/81\\
\hline
1&2&$1/\sqrt{3}$&$1/6$\\
\hline
2&2&$1/\sqrt{3}$&8/9\\
\hline
1&3&$1/\sqrt[3]{3}$&$\left(3-\sqrt[3]{3}\right)^2/(24 \sqrt[3]{3})$\\
\hline
2&3&$1/\sqrt[3]{3}$&$\left(9-3^{2/3}\right)^2/(72\times 3^{2/3})$\\
\hline
\end{tabular}
\caption{$r_k\in(0,1)$ is the unique positive root of the equation $3r^k-1=0$}
\label{tab3}\end{table}
In the following, we deduce the sharp improved version of the Bohr inequality involving two Schwarz functions.
\begin{theo}\label{Th3}
Suppose that $f(z)=\sum_{n=0}^\infty a_n z^n$ is analytic in $\D$ with $|f(z)|\leq 1$ in $\D$ and $\omega_k\in\mathcal{B}_k$ for $k\geq 1$. Then
\beas   |f(\omega_m(z))|+\sum_{n=2}^\infty \left|\frac{f^{(n)}(\omega_m(z))}{n!}\right||\omega_k(z)|^n\leq 1\eeas 
for $r\leq R_{3,m,k}$, where $R_{3,m,k}\in(0,1)$ is the smallest positive root of the equation 
\beas 2r^{2k}-(1-r^{2m})(1-r^k-r^m)=0.\eeas
The number $R_{3,m,k}$ cannot be improved.
\end{theo} 
\begin{proof}
Let $F_3(x)=1/((1-x)^{n-1}(1-x^2))$, where $0\leq x\leq x_0(\leq1)$ and $n\geq 2$. Then, $F_3'(x)=(n-1+x+nx)/((1-x)^{n+1} (1+x)^2)\geq 0$. Therefore, $F_3'(x)$ is 
monotonically increasing functions of $x$, and it follows that $F_3(x)\leq F_3(x_0)$.
Using similar argument as in the proof of \textrm{Theorem \ref{Th1}} and in view of \textrm{Lemma \ref{lem2}}, we obtain (\ref{f3}) and 
\be\label{f6} \frac{\left|f^{(n)}(\omega_m(z))\right|}{n!}\leq \frac{1-|f(\omega_m(z))|^2}{(1-|\omega_m(z)|)^{n-1}(1-|\omega_m(z)|^2)}\leq \frac{1-|f(\omega_m(z))|^2}{(1-r^m)^{n-1}(1-r^{2m})},\;\; n\geq 2.\ee
Therefore,
\beas && |f(\omega_m(z))|+\sum_{n=2}^\infty \left|\frac{f^{(n)}(\omega_m(z))}{n!}\right||\omega_k(z)|^n\\[2mm]
&\leq& |f(\omega_m(z))|+\sum_{n=2}^\infty \frac{1-|f(\omega_m(z))|^2}{(1-r^m)^{n-1}(1-r^{2m})}r^{kn}\\[2mm]
&=&|f(\omega_m(z))|+ \frac{1-|f(\omega_m(z))|^2}{(1+r^{m})}\sum_{n=2}^\infty \left(\frac{r^k}{1-r^m}\right)^n\eeas
\beas&=&|f(\omega_m(z))|+\frac{r^{2 k}}{(1-r^{2m}) (1-r^k-r^m)}\left(1-|f(\omega_m(z))|^2\right)\\[2mm]
&\leq& \frac{a+r^m}{1+ar^m}+\frac{r^{2 k}}{(1-r^{2m}) (1-r^k-r^m)}\left(1-\left(\frac{a+r^m}{1+ar^m}\right)^2\right)\\[2mm]
&=&1+\frac{(1-a)G_9(a,r)}{(1+ar^m)},\eeas
where the second inequality holds for such $r\in[0,1]$ satisfying $0\leq r^{2 k}/((1-r^{2m}) (1-r^k-r^m))\leq 1/2$, {\it i.e.,} for $r\in[0, r_{3,m,k}]$, where $r_{3,m,k}\in(0,1)$ is the smallest positive root of the equation $2r^{2 k}-(1-r^{2m}) (1-r^k-r^m)=0$ and 
\beas G_9(a,r)=\frac{(1+a)r^{2k}}{(1-r^k-r^m)(1+ar^m)}-(1-r^m).\eeas
Differentiating partially $G_9(a,r)$ with respect to $a$, we obtain
\beas \frac{\pa}{\pa a}G_9(a,r)=\frac{(1-r^m)r^{2k}}{(1-r^k-r^m)(1+a r^m)^2}\geq0,\eeas
which shows that $G_9(a,r)$ is a monotonically increasing function of $a\in[0,1]$ and it follows that 
\beas G_9(a,r)\leq G_9(1,r)=\frac{2r^{2k}}{(1-r^k-r^m)(1+r^m)}-(1-r^m)\leq 0\eeas
for $r\leq R_{3,m,k}$, where $R_{3,m,k}\in(0,1)$ is the smallest positive root of the equation 
\bea\label{f7} 2r^{2k}-(1-r^{2m})(1-r^k-r^m)=0. \eea
Note that $r_{3,m,k}=R_{3,m,k}$. \\[2mm]
\indent To prove the sharpness of the result, we consider the following functions 
\beas f_4(z)=\frac{a-z}{1-az}=A_0+\sum_{n=1}^\infty A_n z^n\quad\text{and}\quad \omega_m(z)=z^m\quad\text{for}\quad m\geq 1,\eeas
where $A_0=a\in[0,1)$, $A_n=-(1-a^2) a^{n-1}$ for $n\geq 1$. Note that $f_4^{(n)}(z)/n!=a^{n-1}(1-a^2)/(1-ar)^{n+1}$.
Thus, for $z=r<\sqrt[m]{a}$, we have 
\beas |f_4(\omega_m(z))|+\sum_{n=2}^\infty \left|\frac{f_4^{(n)}(\omega_m(z))}{n!}\right||\omega_k(z)|^n
&=&\frac{a- r^m}{1-a r^m}+\frac{(1-a^2)}{a(1-ar^m)}\sum_{n=2}^\infty \left(\frac{ar^{k}}{1-ar^m}\right)^n\\[2mm]
&=&1+\frac{(1-a)G_{10}(a, r)}{1-ar^m},\eeas
where 
\beas G_{10}(a, r)=-(1+r^{m})+\frac{a(1+a) r^{2 k}}{(1 -a r^m) (1-a r^k -a r^m)}.\eeas
It is evident that
\beas \lim_{a\to1^-} G_{10}(a, r)=-(1+r^{m})+\frac{2r^{2 k}}{(1-r^m) (1-r^k- r^m)}>0 \eeas
for $r>R_{3,m,k}$, where $R_{3,m,k}$ is the smallest positive root of the equation (\ref{f7}). Thus the radius $R_{3,m,k}$ is the best possible. This completes the proof.
\end{proof}
\begin{rem}
Setting $k=m=1$ and $\omega_1(z)=z$ in \textrm{Theorem \ref{Th3}} gives \textrm{Theorem 2.2} of \cite{LSX2018} for $N=2$. 
\end{rem}
\noindent In Table \ref{tab3} and Figure \ref{fig5}, we obtain the values of $R_{3,m,k}$ for certain values of $m,k\in\mathbb{N}$.
\begin{table}[H]
\centering
\begin{tabular}{*{7}{|c}|}
\hline
$m$ &1&2&2&3&4&10\\
\hline
k&1&1&2&2&10&15\\
\hline
$R_{3,m,k}$&0.355416&0.430586&0.596168&0.633513&0.869519&0.924302\\
\hline
\end{tabular}
\caption{$R_{3,m,k}$ is the smallest positive root of the equation (\ref{f7}) in $(0, 1)$}
\label{tab3}\end{table}
\begin{figure}[H]
\centering
\includegraphics[scale=0.8]{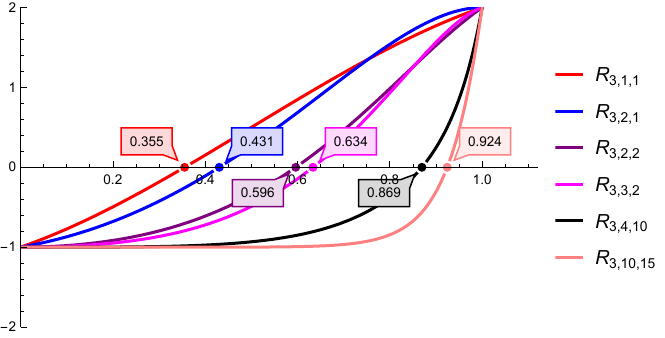}
\caption{Location of $R_{3,m,k}$ in $(0,1)$}
\label{fig5}
\end{figure}
\noindent In the following result, we establish the sharp improved version of the classical Rogosinski inequality containing two Schwarz functions. 
\begin{theo}\label{Th5}
Suppose that $f(z)=\sum_{n=0}^\infty a_n z^n$ is analytic in $\D$ with $|f(z)|\leq 1$ in $\D$ and $\omega_k\in\mathcal{B}_k$ for $k\geq 1$. Then for $N\geq 1$, we have
\beas |f(\omega_m(z))|+\sum_{n=1}^N \left|\frac{f^{(n)}(\omega_m(z))}{n!}\right||\omega_k(z)|^n\leq 1\eeas 
for $r\leq R_{4,m,k}$, where $R_{4,m,k}\in(0,1)$ is the smallest positive root of the equation 
\beas\frac{2r^{k}(1-r^m)^N-2r^{k(N+1)}}{(1-r^m)^{N-1}(1-r^k-r^m)(1+r^m)}-(1-r^m)=0.\eeas
The number $R_{4,m,k}$ cannot be improved.
\end{theo}
\begin{proof}Using similar argument as in the proof of \textrm{Theorems \ref{Th1}} and \ref{Th3}, and in view of \textrm{Lemma \ref{lem2}}, we obtain (\ref{f3}), (\ref{f6}) and $|a_n|\leq 1-a^2$ for $n\geq 1$.
Therefore,
\beas && |f(\omega_m(z))|+\sum_{n=1}^N \left|\frac{f^{(n)}(\omega_m(z))}{n!}\right||\omega_k(z)|^n\\[2mm]
&\leq& |f(\omega_m(z))|+\sum_{n=1}^N \frac{1-|f(\omega_m(z))|^2}{(1-r^m)^{n-1}(1-r^{2m})}r^{kn}\\[2mm]
&=&|f(\omega_m(z))|+\frac{r^k(1 - r^m)^N -r^{k(N+1)}}{(1+r^m)(1-r^m)^{N}(1 - r^k - r^m)}\left(1-|f(\omega_m(z))|^2\right)\eeas
\beas &\leq& \frac{a+r^m}{1+ar^m}+\frac{r^k(1 - r^m)^N -r^{k(N+1)}}{(1+r^m)(1-r^m)^{N}(1 - r^k - r^m)}\left(1-\left(\frac{a+r^m}{1+ar^m}\right)^2\right)\\[2mm]
&=&1+\frac{(1-a)H_1(a,r)}{(1+ar^m)},\eeas
where 
\beas H_1(a,r)=\frac{(1+a)(r^k(1 - r^m)^N -r^{k(N+1)})}{(1-r^m)^{N-1}(1 - r^k - r^m)(1+ar^m)}-(1-r^m)\eeas
 and the second inequality holds for such $r\in[0,1]$ satisfying
\beas 0\leq \frac{r^k(1 - r^m)^N -r^{k(N+1)}}{(1+r^m)(1-r^m)^{N}(1 - r^k - r^m)}\leq \frac{1}{2},\eeas 
{\it i.e.,} for $r\in[0, r_{4,m,k}]$, where $r_{4,m,k}\in(0,1)$ is the smallest positive root of the equation 
\beas 2r^{k}(1-r^m)^N-2r^{k(N+1)}-(1+r^m)(1-r^m)^{N}(1-r^k-r^m)=0.\eeas
It is evident that $H_1(a,r)$ is a monotonically increasing function of $a\in[0,1]$ and it follows that 
\beas H_1(a,r)\leq H_1(1,r)=\frac{2r^{k}(1-r^m)^N-2r^{k(N+1)}}{(1-r^m)^{N-1}(1-r^k-r^m)(1+r^m)}-(1-r^m)\leq 0 \eeas
for $r\leq R_{4,m,k}$, where $R_{4,m,k}$ is the smallest positive root of the equation 
\bea\label{f8}\frac{2r^{k}(1-r^m)^N-2r^{k(N+1)}}{(1-r^m)^{N-1}(1-r^k-r^m)(1+r^m)}-(1-r^m)=0.\eea
Note that $R_{4,m,k}=r_{4,m,k}$.\\[2mm]
\indent To prove the sharpness of the result, we consider the following functions 
\beas f_5(z)=\frac{a-z}{1-az}=A_0+\sum_{n=1}^\infty A_n z^n\quad\text{and}\quad \omega_m(z)=z^m\quad\text{for}\quad m\geq 1,\eeas
where $A_0=a\in[0,1)$, $A_n=-(1-a^2) a^{n-1}$ for $n\geq 1$.
Thus, for $z=r<\sqrt[m]{a}$, we have 
\beas|f_5(\omega_m(z))|+\sum_{n=1}^N \left|\frac{f_5^{(n)}(\omega_m(z))}{n!}\right||\omega_k(z)|^n
&=&\frac{a- r^m}{1-a r^m}+\frac{(1-a^2)}{a(1-ar^m)}\sum_{n=1}^N \left(\frac{ar^{k}}{1-ar^m}\right)^{n}\\[2mm]
&=&1+\frac{(1-a)H_2(a, r)}{(1-ar^m)},\eeas
where
\beas H_2(a, r)=-(1+r^{m})+(1+a)\frac{r^k(1-ar^m)^N-a^Nr^{k(N+1)}}{(1-ar^m)^{N}(1-ar^k-a r^m)}.\eeas
It is evident that
\beas \lim_{a\to1^-} H_2(a, r)=\frac{(1+r^{m})}{(1-r^{m})}\left(\frac{2r^{k}(1-r^m)^N-2r^{k(N+1)}}{(1-r^m)^{N-1}(1-r^k-r^m)(1+r^m)}-(1-r^m)\right)>0 \eeas
for $r>R_{4,m,k}$, where $R_{4,m,k}$ is the smallest positive root of the equation (\ref{f8}). Thus the radius $R_{4,m,k}$ is the best possible. This completes the proof.
\end{proof}
\begin{rem}
We discuss some special cases of \textrm{Theorem \ref{Th5}}.
\begin{enumerate}
\item[(i)] Setting $k=1$, $m\to\infty$ with $\omega_m(z)=z^m$ in \textrm{Theorem \ref{Th5}}, we see that $R_{4,m,k}\to 1/2$. Hence, we have $\sum_{n=0}^N |a_n|z^n\leq 1$ for $r\leq 1/2$ and $N\geq 1$, which is the classical Rogosinski inequality.
\item[(ii)] Setting $k=1$, $m\to\infty$ and $N\to\infty$ with $\omega_m(z)=z^m$ in \textrm{Theorem \ref{Th5}}, we see that $R_{4,m,k}\to 1/3$. Hence, we have $\sum_{n=0}^\infty |a_n|z^n\leq 1$ for $r\leq 1/3$, which is the classical Bohr inequality.
\end{enumerate}
\end{rem}
In the following result, we establish the sharp improved version of the Bohr inequality involving two Schwarz functions.
\begin{theo}\label{Th4}
Suppose that $f(z)=\sum_{n=0}^\infty a_n z^n$ is analytic in $\D$ with $|f(z)|\leq 1$ in $\D$ and $\omega_k\in\mathcal{B}_k$ for $k\geq 1$. Then
\beas   |f(\omega_m(z))|^2+\sum_{n=2}^\infty \left|\frac{f^{(n)}(\omega_m(z))}{n!}\right||\omega_k(z)|^n\leq 1\eeas 
for $r\leq R_{5,m,k}$, where $R_{5,m,k}\in(0,1)$ is the smallest positive root of the equation $r^{2 k}-(1-r^{2m}) (1-r^k-r^m)=0$.
The number $R_{5,m,k}$ cannot be improved.
\end{theo} 
\begin{proof}
Using similar argument as in the proof of \textrm{Theorems \ref{Th1}} and \ref{Th3}, and in view of \textrm{Lemma \ref{lem2}}, we obtain (\ref{f3}) and (\ref{f6}).
Therefore
\beas && |f(\omega_m(z))|^2+\sum_{n=2}^\infty \left|\frac{f^{(n)}(\omega_m(z))}{n!}\right||\omega_k(z)|^n\\[2mm]
&\leq&|f(\omega_m(z))|^2+\frac{r^{2 k}}{(1-r^{2m}) (1-r^k-r^m)}\left(1-|f(\omega_m(z))|^2\right)\\[2mm]
&\leq& 1+\left(1-\frac{r^{2 k}}{(1-r^{2m}) (1-r^k-r^m)}\right)\left(\left(\frac{a+r^m}{1+ar^m}\right)^2-1\right)\\[2mm]
&=&1-\frac{(1-a^2)\left((1-r^{2m}) (1-r^k-r^m)-r^{2 k}\right)}{(1-r^k-r^m)(1+ar^m)^2}\\[2mm]
&=&1-\frac{(1-a^2)H_3(a,r)}{(1+ar^m)^2},\eeas
where the second inequality holds for such $r\in[0,1]$ satisfying $r^{2 k}/((1-r^{2m}) (1-r^k-r^m))\leq 1$, {\it i.e.,} for $r\in[0, R_{5,m,k}]$, where $R_{5,m,k}\in(0,1)$ is the smallest positive root of the equation $r^{2 k}-(1-r^{2m}) (1-r^k-r^m)=0$ and 
\beas H_3(a,r)=\frac{(1-r^{2m}) (1-r^k-r^m)-r^{2 k}}{(1-r^k-r^m)}\geq 0\quad\text{for}\quad r\leq R_{5,m,k}.\eeas
\indent To prove the sharpness of the result, we consider the following functions 
\beas f_6(z)=\frac{a-z}{1-az}=A_0+\sum_{n=1}^\infty A_n z^n\quad\text{and}\quad \omega_m(z)=z^m\quad\text{for}\quad m\geq 1,\eeas
where $A_0=a\in[0,1)$, $A_n=-(1-a^2) a^{n-1}$ for $n\geq 1$.
Thus, for $z=r<\sqrt[m]{a}$, we have 
\beas|f_6(\omega_m(z))|^2+\sum_{n=2}^\infty \left|\frac{f_6^{(n)}(\omega_m(z))}{n!}\right||\omega_k(z)|^n
&=&\left(\frac{a- r^m}{1-a r^m}\right)^2+\sum_{n=2}^\infty \frac{a^{n-1}(1-a^2)r^{kn}}{(1-ar^m)^{n+1}}\\[2mm]
&=&1+\frac{(1-a^2)H_4(a, r)}{(1-ar^m)^2},\eeas
where 
\beas H_4(a, r)=-(1+r^{2m})+\frac{ar^{2 k}}{(1-a r^k -a r^m)}.\eeas
It is evident that
\beas \lim_{a\to1^-} H_4(a, r)=-(1+r^{2m})+\frac{r^{2 k}}{(1-r^k- r^m)}>0 \eeas
for $r>R_{5,m,k}$, where $R_{5,m,k}$ is the smallest positive root of the equation $r^{2 k}-(1-r^{2m}) (1-r^k-r^m)=0$. Thus the radius $R_{5,m,k}$ is the best possible. This completes the proof.
\end{proof}
\begin{rem}
Setting $k=m=1$ and $\omega_1(z)=z$ in \textrm{Theorem \ref{Th3}} gives \textrm{Corollary 2.3} of \cite{LSX2018} for $N=2$. 
\end{rem}
\noindent In Table \ref{tab4} and Figure \ref{fig6}, we obtain the values of $R_{5,m,k}$ for certain values of $m,k\in\mathbb{N}$.
\begin{table}[H]
\centering
\begin{tabular}{*{7}{|c}|}
\hline
$m$ &1&2&2&3&4&10\\
\hline
k&1&1&2&2&10&15\\
\hline
$R_{5,m,k}$&0.403032&0.49478&0.634848&0.676754&0.880073&0.931868\\
\hline
\end{tabular}
\caption{$R_{5,m,k}$ is the smallest positive root of the equation $r^{2 k}-(1-r^{2m}) (1-r^k-r^m)=0$ in (0, 1)}
\label{tab4}\end{table}
\begin{figure}[H]
\centering
\includegraphics[scale=0.9]{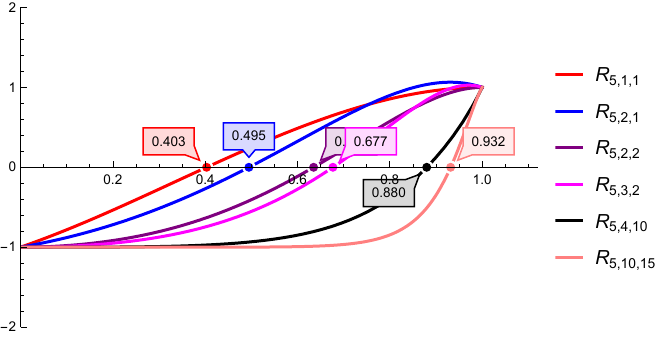}
\caption{Location of $R_{5,m,k}$ in $(0,1)$}
\label{fig6}
\end{figure}
\section*{Declarations}
\noindent{\bf Acknowledgment:} The work of the second author is supported by University Grants Commission (IN) fellowship (No. F. 44 - 1/2018 (SA - III)).\\[2mm]
{\bf Conflict of Interest:} The authors declare that there are no conflicts of interest regarding the publication of this paper.\\[1mm]
{\bf Availability of data and materials:} Not applicable.\\[1mm]
{\bf Authors' contributions:} All authors contributed equally to the investigation of the problem, and all authors have read and approved the final manuscript.

\end{document}